\documentclass[runningheads, 12pt]{amsart}

\usepackage{graphicx}
\usepackage{amssymb}
\usepackage{amsthm}
\usepackage{amsmath}

\newcommand{\be}{\begin{eqnarray}}
\newcommand{\ee}{\end{eqnarray}}

\newcommand{\bn}{\begin{eqnarray*}}       
\newcommand{\en}{\end{eqnarray*}}
\newtheorem{theorem}{Theorem}
\newtheorem{lemma}{Lemma}

\newtheorem{remark}{Remark}

\newcommand{\bea}{\begin{eqnarray*}}
\newcommand{\eea}{\end{eqnarray*}}
\newcommand{\ben}{\begin{eqnarray}}
\newcommand{\een}{\end{eqnarray}}
\newcommand{\beq}{\begin{equation}}
\newcommand{\eeq}{\end{equation}}

\newcommand{\R}{\ensuremath{\mathbb{R}}}

\newcommand{\eps}{\varepsilon}
\newcommand{\de}{\delta}
\newcommand{\si}{\sigma}

\newcommand{\pint}{PV\!\!\!\int}

\newcommand{\om}{\omega}
\renewcommand{\d}{\partial}

\newcommand{\te}{\theta}

\newcommand{\cT}{\mathcal{G}}
\newcommand{\ka}{\kappa}

\begin{document}

\title{Cusp formation for a nonlocal evolution equation}


\author{Vu Hoang         \and Maria Radosz}




\maketitle

\begin{abstract}
In this paper, we introduce a nonlocal evolution equation inspired by the C\'ordoba-C\'ordoba-Fontelos nonlocal transport equation. 
The C\'ordoba-C\'ordoba-Fontelos equation can be regarded as a model for the 2D surface quasigeostrophic equation or the Birkhoff-Rott equation.
We prove blowup in finite time, and more importantly, investigate conditions under which the solution forms a cusp in finite time.
\end{abstract}

\section{Introduction}
\label{intro}
The non-linear, non-local active scalar equation 
\beq\label{CCF}
\te_t + u \te_y = 0, ~~ u = H\te
\eeq
for $\te = \te(y, t)$ was proposed by A. C\'ordoba, D. C\'ordoba and A. Fontelos \cite{CCF} as a one-dimensional analogue of the 
2D surface quasigeostrophic equation (SQG) \cite{constantin1994formation}
and the Birkhoff-Rott equation \cite{Saffmann}. We refer to \eqref{CCF} as the CCF equation.

The study of one-dimensional equations modeling various aspects of three- and two-dimensional fluid mechanics problems has a
long-standing tradition (\cite{Castro,CLM, HouLuo1,HouLi}). 

It is well-known that smooth solutions of \eqref{CCF} lose their regularity in finite time (\cite{CCF,SashaRev,silvestre2014transport}). 
However, little is understood about the precise way in which the singularity forms. A particularly simple scenario 
is as follows: We consider smooth, even solutions, i.e. $\te(-y, t) = \te(y, t)$ such that
the initial condition $\te_0(y)$ has a single nondegenerate maximum and decays sufficiently quickly at infinity.
Numerical simulations \cite{CCF, silvestre2014transport} seem to indicate that the solution forms a \emph{cusp} at
the singular time $T_s$, so that 
\beq\label{universalCusp}
\te(y, T_s) \sim \te_0(0)- \text{const.}|y|^{1/2}
\eeq
close to the origin. In \cite{silvestre2014transport}, the authors make the conjecture that in a generic sense,
all maxima eventually develop into cusps of the form \eqref{universalCusp}.

Besides the original blowup proof \cite{CCF}, various proofs of blowup for \eqref{CCF} have been found recently \cite{silvestre2014transport}.
In \cite{Tam}, a discrete model for \eqref{CCF} was studied.

Many of the known proofs work when an additional viscosity term is present in \eqref{CCF}. On the other hand none of them explains the cusp formation since
the shape of the solution is not controlled at the singular time. The task of establishing that solutions of \eqref{CCF}
exhibit cusp formation appears to be challenging.

We believe that making progress on the question of cusp formation is important for the following reasons. In order to prove cusp formation
for \eqref{CCF}, we need to develop insight into the intrinsic mechanism of singularity formation. This mechanism is not well-understood
and seems to favor blowup at certain predetermined points. Moreover, the solution apparently remains regular outside of these points
- even at the time of blowup.

This situation is similar to recent in numerical observations for the three-dimensional Euler equations
by T. Hou and G. Luo \cite{HouLuo1}. There, the authors exhibit solutions for which the magnitude of the vorticity
vector appears to become infinite at an intersection point of a domain boundary with a symmetry axis.
Their blowup scenario is referred to as the \emph{hyperbolic flow scenario}.
In \cite{HouLuo2}, a 1D model problem for axisymmetric flow was introduced, and finite time blowup
was proven in \cite{sixAuthors}. We would also like to mention \cite{Pengfei}, in which the authors study a 
simplified model (proposed in \cite{CKY}) of the equation introduced by T. Hou and G. Luo in \cite{HouLuo2}, 
proving the existence of self-similar solutions using computer-assisted means.

A sufficiently detailed understanding of the singularity formation mechanism for nonlocal active scalar equations
like \eqref{CCF} is very likely a prerequisite for obtaining a blowup proof for the 3D Euler equations.  
We refer to \cite{ConstRev} for a thorough discussion of this outstanding and challenging problem.  

In this paper, our goal is to lay groundwork for future investigation by studying a nonlocal
active scalar equation for which the singular behavior of the solution at the blowup time can be characterized.
We develop new techniques that allow us to obtain some control on the singular shape.  
Our model problem is related to \eqref{CCF} and reads as follows:
\beq\label{outerModel}
\begin{split}
&\te_t + u \te_y = 0, \\
&u_y (y, t) = \int_{y}^\infty \frac{  \te_y(z, t)}{z}~dz, ~~u(0, t) = 0.
\end{split}
\eeq
We consider solutions $\te$ defined on $[0, \infty)$ and think of them as being
extended to $\R$ as even functions. 
Our main result, Theorem \ref{blowupThm}, states that solutions blow up in finite time, either
forming a cusp or a needle-like singularity. To the best of our knowledge, this is the first
time such a scenario has been rigorously established.

Our paper is structured as follows: in the remaining part of this introductory section, we motivate the introduction of
\eqref{outerModel}. In section \ref{mainRes}, we state our main results.
The remainder of our paper provides the proofs of Theorems 
\ref{thmWellposed} and \ref{blowupThm}.

\subsection{Derivation of the model equation.}
An essential issue with nonlocal transport equations is predicting from the outset where the 
singularity will form. The even symmetry of $\te$ helps, by creating a stagnant
point of the flow at the origin. At this stage of our understanding of blowup scenarios, however, this still does not rule out 
the possibility of a singularity also forming somewhere else. 
Our goal is to simplify by writing down a model where singularities can only form at a given point,
following ideas of ``hyperbolic approximation"  or ``hyperbolic cut-off" 
from \cite{kiselev2013small} and \cite{CKY}.  We first describe
how this applies to \eqref{CCF}.

The velocity gradient of equation \eqref{CCF} is given by $u_y = H\te_y$. Using the odd symmetry of
$\te_y$, we can write
\bea
u_y(y, t) = \pint_0^{2 y} \frac{2 z \te_y(z, t)\, dz}{z^2-y^2} + \int_{2 y}^\infty \frac{2 z \te_y(z, t)\, dz}{z^2-y^2}. 
\eea
In a first approximation we only retain the long-range part of the interaction, 
and therefore consider the model Biot-Savart law
\beq\label{cutOff}
u_y(y, t) = \int_{y}^\infty \frac{  \te_y(z, t)\, dz}{z}, ~~u(0, t) = 0.
\eeq
The long-range part of the interaction has been approximated by shifting the 
kernel singularity to the origin (and we have also dropped the non-essential factor $2$).
In general, the hyperbolic cut-off emphasizes the non-local role of the fluid surrounding the singular point
as a leading part of the blowup mechanism. Due to certain monotonicity properties,
the intrinsinc blowup mechanism becomes transparent (see Remark \ref{expl}).

Our main inspiration to derive \eqref{outerModel} came from \cite{kiselev2013small}, where 
the hyperbolic flow scenario for the 2D Euler equation on a disc was considered. 
There, a useful new representation of the velocity field for the 2D Euler equation was discovered. Briefly,
the velocity field close to a stagnant point of the flow at the intersection of a domain boundary and a symmetry
axis of the solution was decomposed into a main part and error term
$$
u_1(x_2, x_2, t) = - x_1\iint_{y_1\geq x_1, y_2\geq x_2}\frac{y_1 y_2}{|y|^4}\om(y, t) ~dy_1 dy_2 + x_1 B_1(x_1, x_2, t).
$$
The error term $B_1$ could be controlled in a certain sector. Note that, as in \eqref{cutOff},  the main part of the
velocity field is given by an integral over a kernel with singularity at $y = (0, 0)$.

Finally we note that K.~Choi, A.~Kiselev and Y.~Yao use a similar process in \cite{CKY} 
to approximate a one-dimensional model of 
an equation proposed by T.~Hou and G.~Luo in \cite{HouLuo2}.

\section{Main results}\label{mainRes}
We prove that for the model \eqref{outerModel}, solutions blow up in
finite time for generic data, and that the singularity is either a cusp or a needle-like
singularity.  
We are looking for solutions of \eqref{outerModel} in the class 
\beq
\theta(y, t) \in C( [0, T),\,C_0^2(\R^+))\cap C^1([0, T), \, C^1(\R^+)), ~~\theta_y(0, t) = 0,
\eeq
where $\R^+ = [0, \infty)$ and $C_0^2(\R^+)$ denotes twice continuously differentiable
functions with compact support in $\R^+$.
This is a natural class for solutions of \eqref{outerModel} which have a single maximum at the origin. 

Our results are as follows: 
\begin{theorem}\label{thmWellposed}
The problem \eqref{outerModel} is locally well-posed for compactly supported initial data 
\beq\nonumber
\theta_0 \in C_0^2(\R^+), ~~\d_y\te_0(0)=0.
\eeq
\end{theorem}

\begin{theorem}\label{blowupThm}
Assume the initial data $\te_0$ is compactly supported, nonincreasing, nonnegative and is such that
\beq\nonumber
\d_{yy}\te_0(0) < 0.
\eeq
 Then there exists a finite time $T_s>0$ and constants
$\nu\in (0, 1), c > 0$ depending on the inital data 
such that $\theta(y, t) \in C^1( [0, T_s),\,C_0^2(\R^+))$ and
$$\te(y, T_s)=\lim_{t \to T_s} \te(y, t)$$ 
exists for all $y\in \R^+$ and 
\bea
\te(\cdot, T_s)\in C^2(0, \infty).
\eea
Moreover,
\beq\label{cuspBound}
\te(y, T_s) \leq \te_0(0)  - c |y|^{\nu} \quad (y\in \R^+).
\eeq 
\end{theorem}
The singularity formed at time $T_s$ is at least a cusp, but can potentially also be ``needle-like" (see Figure \ref{f1}). Needle formation arises when
$$
\lim_{y\to 0+} \te(y, T_s) < \te_0(0).
$$
To the best of our knowledge, 
this is the first time singularity formation of the type described is rigourously established for a nonlocal active
scalar transport equation.

\begin{figure}[htbp]
\begin{center}
\includegraphics[scale=0.6]{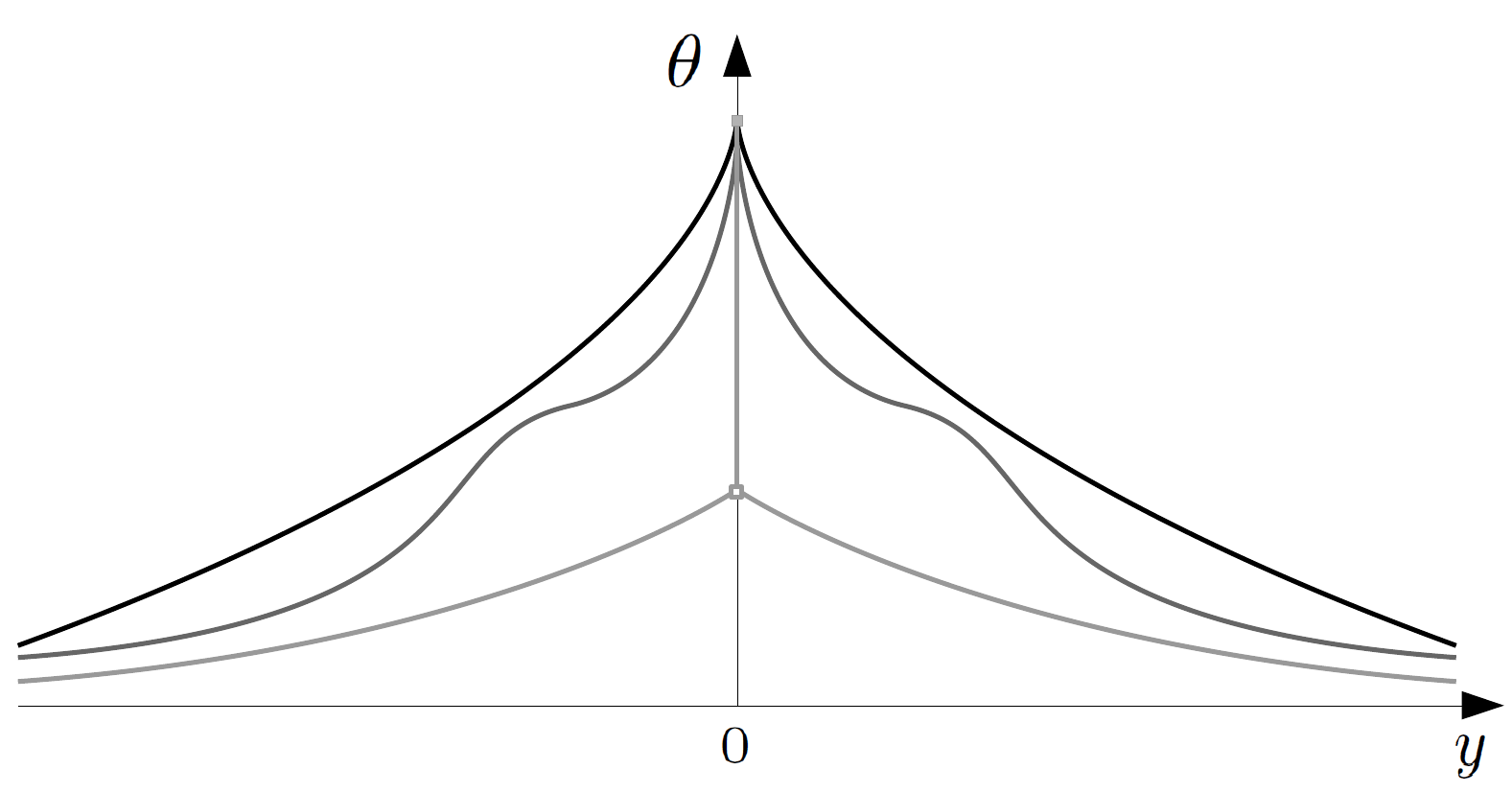}%
\end{center}
\caption{Illustration of the upper bound (black), $\te(\cdot, T_s)$ at singular time with cusp (grey) and a scenario with ``needle formation" (light grey). Note that all three have the same value at $y=0$.
 }\label{f1}
\end{figure}

\section{Proofs}
\subsection{An equation for the flow map.}
Our approach is characterized by working in Langrangian coordinates and exploiting the properties of the flow map.
From now on, $x$ denotes the Lagrangian space variable. The flow map $\Phi$ associated to \eqref{outerModel} satisfies
\beq\label{eq_flowmap}
\Phi_t(x, t) = u(\Phi(x, t), t).
\eeq
Note that because of $u(0,t)=0$, $\Phi(0,t)=0$ holds.
The basic equation for the stretching (derivative of the flow map $\phi = \partial_{x}\Phi$) follows from differentiating \eqref{eq_flowmap}: 
\beq\label{basic_Stretch}
\phi_t(x, t) = u_y(\Phi(x, t), t) \phi(x, t).
\eeq
With some nonnegative, compactly supported, smooth $g: [0, \infty)\to \R$, $g\not \equiv 0$, one can make an ansatz for the solution $\te$ by setting
\beq\label{def_te}
\theta(y, t) = g(\Phi^{-1}(y, t)).
\eeq

The following relation holds: 
\beq\label{rel:basic}
\te_y(y, t) = g'(\Phi^{-1}(y, t))/\phi(\Phi^{-1}(y, t)).
\eeq
Computing $u_y(\Phi(x, t), t)$ using \eqref{cutOff}, \eqref{rel:basic} and the substitution $z = \Phi^{-1}(y, t)$ gives
\beq\label{rel_u}
u_y(\Phi(x, t), t) = \int_{\Phi(x, t)}^\infty \frac{g'(\Phi^{-1}(y, t))}{y ~\phi(\Phi^{-1}(y, t))}\, dy = \int_x^{\infty}  \frac{g'(z)}{\Phi(z,t)}\, dz.
\eeq
Combining this with \eqref{basic_Stretch}, we obtain the following central evolution equation for $\phi$:
\begin{align}\label{main_eq}
\begin{split}
&\phi_t(x, t) = \phi(x, t) \int_x^{\infty} \frac{g'(z) ~dz}{\Phi(z, t)}, \qquad\Phi(z, t) = \int_0^z \phi(\si, t)\, d\si.\\
&\phi(x, 0) = \phi_0(x)
\end{split}
\end{align}
Here, $(\phi_0, g)$ are given functions with $(\phi_0, g)\in C^1(\R^+)\times C^2_0(\R^+)$ having the properties
\beq\label{cond_g}
\inf_{\R^+} \phi_0 > 0, ~~ g'(0) = 0. 
\eeq
We shall consider solutions $\phi\in C^1(\R^+\times[0, T])$ with the additional property 
\beq\label{eq_inf}
\inf_{(x, t)\in \R^+\times [0, T]} \phi(x, t) > 0.
\eeq

The following Lemma clarifies the relation between $\te$ and $(\phi_0, g)$:

\begin{lemma}\label{backToTheta}
Let $\te_0 \in C^2_0(\R^+)$ with $\partial_y \te_0(0) = 0$ be given.
Suppose $$\phi\in C^1(\R^+\times[0, T])$$ is a solution of \eqref{main_eq} with given $(\phi_0, g)$ satisfying 
\eqref{cond_g} and
\beq\label{choice_g}
\te_0(y) = g(\Phi_0^{-1}(y)) \quad (y\in \R^+), ~~\text{where }\Phi_0(x) = \int_0^x \phi_0(\si)\,d\si. 
\eeq
Then $\te(y, t) = g(\Phi^{-1}(y,t))\in C( [0, T),\,C_0^2(\R^+))\cap C^1([0, T), \, C^1(\R^+))$ is a solution of \eqref{outerModel} as long as \eqref{eq_inf} holds.
\end{lemma}

\begin{proof}
Observe first that the definition of $g$ in \eqref{choice_g} implies $g'(0) = 0$ via the assumption $\partial_y \te_0(0) = 0$. 
Let $\te(y, t)$ be defined by  $\te(y, t) = g(\Phi^{-1}(y,t))$ (note that $\Phi^{-1}(\cdot, t)$ is well-defined for
all $t\in [0, T]$ because of \eqref{eq_inf}). A calculation (see \eqref{rel_u}) shows that the integral
$$
\int_x^{\infty}  \frac{g'(z)}{\Phi(z,t)}\, dz
$$
equals $u_y(\Phi(x, t), t)$, where $u$ is the velocity field \eqref{cutOff}. Using \eqref{main_eq}, this implies
$$
\partial_{t}\partial_{x}\Phi(x, t) = \partial_{x}\Phi(x, t) u_y(\Phi(x, t), t).
$$
Taking into account $\Phi(0, t) = 0$ and integrating with respect to $x$, we obtain $\partial_t \Phi(x, t) = u(\Phi(x, t), t)$. By differentiating 
$$\te(\Phi(x,t),t)=g(x)$$
in time we see that $\te$ satisfies equation \eqref{outerModel}. The initial condition $\te(y, 0) = \te_0(y)$ follows
from \eqref{choice_g}. Finally, the required regularity for $\te$ is also straightforward to verify.  
To check for instance $\te(\cdot, t)\in C^2(\R^+)$ we observe first
that \eqref{rel:basic} implies that $\te_y$ is continuous on $\R^+$. Differentiating $\eqref{rel:basic}$, we obtain
$$
\te_{yy}(y, t) = \frac{g''(\Phi^{-1}(y, t))- g'(\Phi^{-1}(y, t)) \phi'(\Phi^{-1}(y, t), t) \phi(\Phi^{-1}(y, t), t)^{-1}}{\phi({\Phi^{-1}(y, t)}, t)^2}
$$
and thus $\te_{yy}\in C(\R^+)$. 
\end{proof}

\begin{figure}[htbp]
\begin{center}
\includegraphics[scale=0.6]{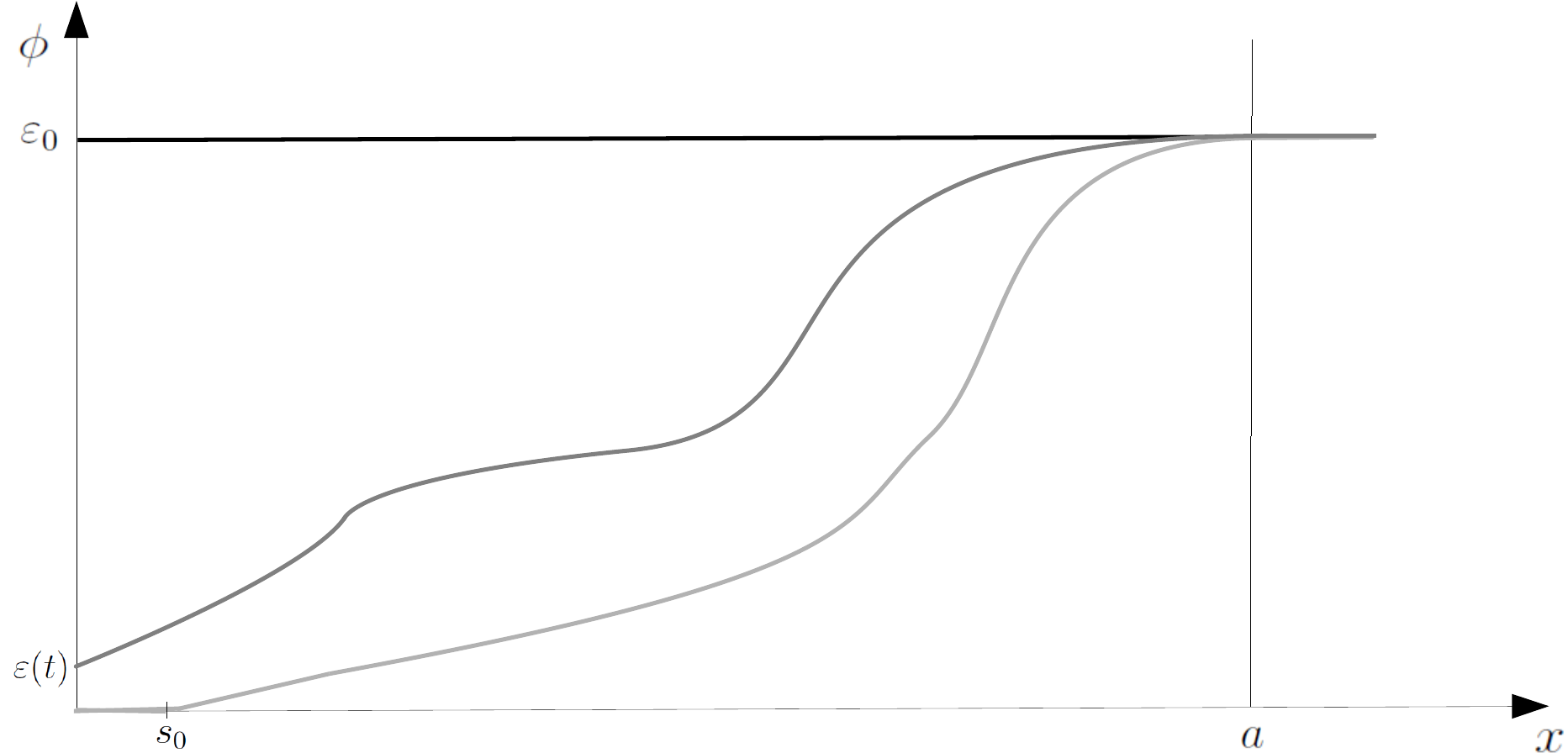}%
\end{center}
\caption{Illustration of $\phi$.
 }\label{f2}
\end{figure}

\begin{remark}\label{expl}
A fairly clear picture of the blowup mechanism emerges by visualizing the solution of \eqref{main_eq}.
Later, we will choose $\phi_0$ to be a positive constant and adjust $g$ so that we obtain solutions of \eqref{outerModel}
by using Lemma \ref{backToTheta}. $\phi(x, t)$ is monotone nonincreasing in $x$ and is depicted in Figure \ref{f2} for times $t>0$.
The blowup happens when $\phi(0, t)$ reaches zero in finite time. The intuitive reason is that at the singular time $T_s$, the odd continuation of the
inverse $\Phi^{-1}(\cdot,T_s)$ to $(-\infty,\infty)$ is not differentiable in $x=0$. The main driving mechanism for blowup in finite time is the behavior of the
following integral: 
$$
-\int_0^\infty \frac{g'(z) ~dz}{\Phi(z, t)}.
$$ 
We shall show that it will be at least growing like $\phi(0, t)^{-1+\beta}$ with some $\beta\in (0, 1)$, thus implying
$\phi_t(0, t) \leq - C \phi(0, t)^{\beta}$. 
\end{remark}

The following Theorem shows that nontrivial solutions of \eqref{main_eq} cannot be
defined on an infinite time interval. We present this Theorem for illustrative purposes, to show that finite-time blowup via
a contradiction argument can easily be obtained for \eqref{main_eq}. The task of characterizing
the singular shape is, however, much more challenging (see section \ref{sec_proof_blowupThm} below).

\begin{theorem}\label{thm:finite}
Suppose \eqref{cond_g} holds for $(\phi_0, g)$ and that $g(0)> 0$.
 There can be no solution of \eqref{main_eq} $\phi\in C^1(\R^+\times[0, \infty))$ satisfying 
\eqref{eq_inf}.
\end{theorem}

\begin{proof}
For the purpose of deriving a contradiction, let us assume that $\phi(\cdot, t)$ is defined for $t\in [0, \infty)$. \eqref{eq_inf}
implies $\Phi(x, t) \geq 0$ for all times. Let $[0, a] = \operatorname{supp} g$. By taking the derivative of 
$\Phi(a, t) = \int_0^a \phi(\si, t)\, d\si$ with respect to $t$ and using \eqref{main_eq} we get
$$
\Phi_t(a, t) = \int_0^a \int_{\si}^{a} \phi(\si, t)\frac{g'(z) }{\Phi(z, t)} ~dz d\si.
$$
The integral on the right-hand side can be written as 
\bea
\int_{0}^a\int_0^z\phi(\si, t)\frac{g'(z) }{\Phi(z, t)}\, d\si dz = \int_0^a \Phi(z, t) \frac{g'(z) }{\Phi(z, t)}\, dz = g(a)-g(0).
\eea
Because of $g(a) = 0$, $\Phi_t(a, t) = - g(0)<0$, giving a contradiction for large times. 
\end{proof}
\subsection{Local existence}

\begin{theorem}[Local existence and uniqueness] \label{locEx}
The problem \eqref{main_eq} has a unique local solution $\phi\in C^1(\R^+\times [0,T])$ for some $T>0$, 
provided that $\phi_0\in C^1(\R^+)$, $g\in C_0^2(\R^+)$ and \eqref{cond_g} holds for $(\phi_0,g)$. Moreover, $$\inf_{(x, t)\in \R^+\times [0, T]} \phi(x, t) > 0.$$
\end{theorem}

\begin{proof} The proof is standard, so we just sketch a few details. Let $\operatorname{supp} g' = [0, a]$.
Rewrite \eqref{main_eq} as an integral equation by integrating \eqref{main_eq} in time
\begin{align}\label{int_eq}
\phi(x,t) & = \phi_0(x)+\int_0^t \phi(x,s)\,\left(\int_x^\infty\frac{g'(z)}{\Phi(z,s)}~dz\right)~ds =:\phi_0(x)+\cT[\phi](x,t).
\end{align}
Given $(g,\phi_0)$ let the set $X_{T,\mu}$ be the set of functions $\phi$ defined by the following conditions:
\begin{align}\nonumber
\phi \in C\left([0,T], C^1(\R^+)\right),~\|\phi-\phi_0\|_{C\left([0,T],C^1[0, a]\right)}\le \mu,~~\phi(x, t) = \phi_0(x)~\text{for}~x\geq a 
\end{align}
where $T,\mu>0$ are to be chosen below. The norm $\|\cdot\|_{C\left([0,T], C^1[0, a]\right)}$
is given by
$$
\max_{t\in [0, T]}\left(\|f(\cdot, t)\|_\infty+\|f_x(\cdot, t)\|_{\infty}\right),
$$
$\|\cdot\|_{\infty}$ denoting the supremum norm on $[0, a]$.

First we need to show that the operator $\phi_0+\cT$ is well-defined.
For $\phi \in X_{T,\mu}$ we have the following estimate:
\begin{align*}
\Phi(z,t) & = \int_0^z \phi(\si,t)~d\si 
\ge -\left|\int_0^z (\phi(\si,t) -\phi_0(\si))~d\si\right|+\int_0^z \phi_0(\si)~d\si\\
& \ge -\mu z+(\displaystyle\min_{\substack{\R^+}}\phi_0)z 
=\left(\displaystyle\min_{\substack{[0,a]}}\phi_0-\mu\right)z.
\end{align*}
So if $\mu=\frac{1}{4}\displaystyle\min_{\substack{\R^+}}\phi_0$, then we have
$$
\frac{g'(z)}{\Phi(z,t)} \leq C(g, \phi_0) \quad (0\leq z \leq a)
$$ 
because of $g'(0) = 0$, i.e. $g'(z)\le \text{const.}z$. Using $\|\d_x\phi\|_\infty\le \mu + \|\d_x\phi_0\|_\infty$, $\|\phi\|_\infty\le \mu+ \|\phi_0\|_\infty$ and 
\eqref{eq_inf} we obtain
\bea
|\cT[\phi](x, t)| \leq  C(g,\phi_0) T,  ~~|\partial_x \cT[\phi](x, t)| \leq  C(g,\phi_0) T,
\eea
and $\cT[\phi](x, t) = 0$ for all $x\geq a, t\in [0, T]$.
Hence $\phi\mapsto\phi_0+\cT[\phi]$ maps $\phi\in X_{T,\mu}$ and into $X_{T, \mu}$ for sufficiently
small $T>0$. A straightforward, but tedious calculation shows the inequality
\begin{align*}
\|\cT[\phi]-\cT[\psi]\|_{C([0,T], C^1[0, a])} & \le C(\mu, g, a, \phi_0) T \|\phi-\psi\|_{C([0,T], C^1[0, a])}.
\end{align*}
The proof is concluded by applying the Contraction Mapping Theorem to the operator $$\phi_0+\cT,$$ choosing $T>0$ to be sufficiently
small to ensure the contraction property. Note that $\phi\in X_{T, \mu}$ satisfying the equation \eqref{int_eq} lies
automatically in $C^1(\R^+\times[0, T])$.
\end{proof}

The proof of Theorem \ref{thmWellposed} follows now from Theorem \ref{locEx}, by taking the local solution
of \eqref{main_eq} and defining $\te$ via \eqref{def_te}. More precisely, we may take e.g. $\phi_0 = 1, 
g(x) = \te_0$.

\subsection{Proof of Theorem \ref{blowupThm}.}\label{sec_proof_blowupThm}

\emph{Setup for the proof and preparatory lemmas.} For convenience, we set
\beq \nonumber
\phi_0(x) = \eps_0>0,
\eeq 
with small $0<\eps_0<1$ to be chosen later, instead of starting with $\phi_0(x) \equiv 1$ (see also Remark \ref{why_eps_0}).
To obtain the given initial condition $\te_0$ we set 
\beq\label{def_g}
g(z; \eps_0) := \te_0(\eps_0 z) \in C^2(\R^+).
\eeq
Note that \eqref{def_g} ensures that $\te$ defined by $\te(y, t) = g(\Phi^{-1}(y,t);\eps_0)$ satisfies the initial
condition $\te(\cdot, 0) = \te_0$ for any $\eps_0>0$. Observe carefully that variation of $\eps_0$ does not correspond to
a rescaling of the initial data for the equation \eqref{outerModel}.

Now fix some constants $0<K_0<K_1$ such that
\beq\label{eq_K1}
\eps_0^2 K_0 z \leq -g'(z; \eps_0) \leq \eps_0^2 K_1 z \quad (0 \leq z \leq 	1)
\eeq
To see that such constants exist, observe that $g''(z; \eps_0) = \eps_0^2 \d_{yy}\te_0(\eps_0 z)$ and $g'(0; \eps_0) = 0$;
just take $K_0$ to be slightly lower and $K_1$ to be slightly larger than $-\d_{yy}\te_0(0)> 0$ and choose $\eps_0$ sufficiently small so that 
\eqref{eq_K1} holds for $0 \leq z \leq 1$. In the following, we write $g(z; \eps_0) = g(z)$ for convenience.


We write 
\beq
\eps(t) = \phi(0, t).
\eeq
By assumption, $\te_y(z, 0) < 0$, so that $g' < 0$. 
Note that \eqref{main_eq} implies $\eps_t(t) < 0$ as long as the smooth solution can be continued.

\begin{lemma}\label{lem_eq_eta}
The following equation holds for $\eta(x, t) = \phi(x, t)/\eps(t)$:
\beq\label{eq_eta}
\eta_t(x, t) = -\eta(x,t) \int_0^x \frac{g'(z)\, dz}{\Phi(z, t)}.
\eeq
Moreover, $\phi(x, t)$ is monotone nondecreasing in $x$ for fixed $t$.
\end{lemma}
\begin{proof} A direct computation gives
\bea 
\eta_t(x, t) = \frac{\eps(t) \phi_t(x, t)- \phi(x, t) \eps_t(t)}{\eps^2(t)} = -\eta(x, t) \int_0^x \frac{g'(z)~dz}{\Phi(z, t)}.
\eea
To see that $\phi$ is monotone in $x$, observe that because of $g'<0$, we have for $x_1 < x_2$
\bea
\phi_t(x_1, t) \leq  \phi(x_1, t) \int_{x_2}^\infty \frac{g'(z)\,dz}{\Phi(z, t)}\leq \frac{\phi(x_1, t)}{\phi(x_2, t)} \phi_t(x_2, t)
\eea
and thus $(\log\phi(x_1, t))_{t}\leq (\log\phi(x_2, t))_{t}$, and the monotonicity follows since $\phi(x_1,0)=\phi(x_2,0)$.
\end{proof}

\begin{lemma}\label{lemma_cont}
A solution $\phi(x, t)$ satisfies the bound $\phi(x, t) \geq \eps(t)$ for all $x\in \R^+$. If $T>0$ is such that $\phi\in C^1(\R^+\times [0, T))$
and $\lim_{t\to T} \eps(t) > 0$, then the solution $\phi$ can be continued to a slightly larger time interval.  
\end{lemma}

\begin{proof}
This follows from Lemma \ref{lem_eq_eta}. From equation \eqref{eq_eta} and $\eta(x, 0) = 1$ we get immediately
$\eta(x, t)\geq 1$, so that actually $\min_{\substack{x\in \R^+}} \phi(x, t) = \eps(t)$. Now apply Theorem \ref{locEx}.
\end{proof}
Let $\beta \in (0, 1)$ and define 
\beq\nonumber
l(t) = \eps^{\beta}(t).
\eeq
Let $\ka>1$ be such that
\beq\label{cond_ka}
\eps_0 \leq \ka \eps_0^{1/\beta}  
\eeq
We consider now the following bootstrap or barrier property \eqref{eq:boot}. 
\beq\tag{B}\label{eq:boot}
\text{$\phi(x, t) \leq \ka \eps(t)$ for all $x\in [0, l(t)]$}  
\eeq 
The continuity of $t\mapsto\phi(\cdot,t)$ implies that \eqref{eq:boot} holds for some short time interval $[0, \tau), \tau>0$.
We now extend the validity of the bootstrap property, with uniform $\ka$, to the whole existence interval 
of our solution $\phi$ by utilizing a kind of \emph{continuous induction}, or \emph{bootstrap argument}, or alternatively speaking a 
\emph{nonlocal maximum principle}.
\begin{itemize}
\item In order to state the argument in a clear fashion, we now let $$\phi\in C^1([0, T_s)\times \R^+)$$
be the solution of \eqref{main_eq} with data $(\eps_0, g(\cdot; \eps_0))$ given by \eqref{def_g},
defined on its maximal existence interval $[0, T_s)$. That is, we continue the local solution $\phi$ for as long
as $\eps(t)>0$. 
\end{itemize}
Note that by Theorem \ref{thm:finite} $T_s$ is finite. However, the in the following we will not use Theorem \ref{thm:finite} and prove blowup in finite time together with a characterization of the solution at the singular time.
Observe that $\lim_{t\to T_s} \eps(t) = 0$.

\begin{lemma}\label{lem:boundInnerZone}
Suppose \eqref{eq:boot} holds on the time interval $[0, T)$, $T>0$. Then 
\beq\nonumber
\eta(x, t) \leq \exp\left( \eps_0^2 K_1 l(t) \int_0^t \eps^{-1}(s)\, ds\right) \quad (x\in [0, l(t)], 0 \leq t < T).
\eeq 
Moreover,
\beq\label{upperBd}
\phi(x, t)\leq \ka x^{1/\beta} \quad (l(t) \leq x < \infty,~0\leq t <T).
\eeq
\end{lemma}
\begin{proof}
We have $\phi(x, t)\geq \eps(t)$, so $\Phi(z, t) \geq \eps(t) z$. By \eqref{eq_eta}, \eqref{eq_K1}, we have for $0 < s < t$
and $x\in [0, l(t)]$:
$$
\eta_t(x,s) \leq \eta(x, s) \eps_0^2 K_1 \frac{l(s)}{\eps(s)}
$$
from which the first statement follows.

For all $l(t) \leq x \leq \eps_0^{\beta}$, let $t_x \leq t$ be the uniquely defined time such that 
$$
l(t_x) = \eps(t_x)^{\beta} = x.
$$
By the bootstrap assumption, $\phi(x, t_x)\leq \ka \eps(t_x) = \ka x^{1/\beta}$. 
From \eqref{main_eq} and the assumption that $g'< 0$ it follows that $\phi(x, t)$ is nonincreasing
in $t$ for fixed $x$. Consequently,
$\phi(x, t) \leq \phi(x, t_x)\leq \ka x^{1/\beta}$ for $t\geq t_x$.

If $x \geq \eps_0^{\beta}$, then we observe $\phi(x, t) \leq \eps_0 \leq \ka \eps_0^{1/\beta} \leq \ka x^{1/\beta}$
by condition \eqref{cond_ka} and also noting $x\geq \eps_0$. 
\end{proof}

The bootstrap assumption also gives a lower bound on $-\eps_t$. In the following Lemma, the structure of the Biot-Savart law of \eqref{outerModel} enters in 
a crucial way.

\begin{lemma}\label{lem:lowerBoundEps}
Suppose \eqref{eq:boot} holds on $[0, T)$. Then
\beq\label{lowerBdEps_t}
- \eps_t (t)\geq \frac{\eps_0^2 K_0 c_\beta}{2\ka} \eps^\beta 
\eeq
where $c_\beta = \frac{\beta (\beta+1)}{(2\beta+1)(1-\beta)}$, provided $\eps_0^{1-\beta} \leq 1/2$.
\end{lemma}
\begin{proof}
For $z\geq l(t)$, using the upper bound \eqref{upperBd} and \eqref{eq:boot} yields
\begin{align*}
\Phi(z, t) &= \Phi(l(t), t) + \int_{l(t)}^z \phi(\si,t)\, d\si
\leq \Phi(l(t), t) + \int_{l(t)}^z \ka \si^{\frac{1}{\beta}}\, d\si\\
&\leq \Phi(l(t), t) + \frac{\ka \beta}{\beta+1} (z^{\frac{1}{\beta} + 1}- l^{\frac{1}{\beta}+1}(t))\le \ka \eps(t) l(t)  + \frac{\ka \beta}{\beta+1} (z^{\frac{1}{\beta} + 1}- l^{\frac{1}{\beta}+1}(t))\\
&\le \ka\left( l^{\frac{1}{\beta}+ 1}(t)   + \frac{\beta}{\beta+1}z^{\frac{1}{\beta} + 1}\right).
\end{align*}
Using this and \eqref{eq_K1}
\begin{align*}
\int_0^\infty \frac{-g'(z) \,dz}{\Phi(z,t)} 
&\ge \frac{\eps_0^2 K_0}{\ka} \int_{l(t)}^1 \frac{z ~dz}{l^{1+\frac{1}{\beta}}(t) + \frac{\beta}{1+\beta} z^{1+\frac{1}{\beta}}}
\ge \frac{\eps_0^2 K_0}{\ka} \int_{l(t)}^1 \frac{z~dz}{(1+ \frac{\beta}{1+\beta}) z^{1+\frac{1}{\beta}}}\\
&= \frac{\eps_0^2 K_0}{\ka} \left(1+\frac{\beta}{1+\beta}\right)^{-1}\frac{\beta}{1-\beta} (l^{1-\frac{1}{\beta}}(t)-1) \\
&= \frac{\eps_0^2 K_0 c_\beta}{\ka} (\eps^{\beta-1}(t)-1).
\end{align*}
Thus
\beq\nonumber
-\eps_t (t) \ge \eps(t) \int_0^\infty\frac{-g'(z)~dz}{\Phi(z, t)} \ge \frac{\eps_0^2 K_0 c_\beta}{\ka}\eps^{\beta}(t) (1-\eps^{1-\beta}(t)) \geq  \frac{\eps_0^2 K_0 c_\beta}{2 \ka}\eps^{\beta}(t)
\eeq
provided $\eps_0^{1-\beta}\leq 1/2$.
\end{proof}

\begin{lemma}\label{lemma_5}
Suppose the following three conditions hold:
\begin{align}\label{threecond}
\begin{split}
\eps_0 &\leq \ka \eps_0^{1/\beta} , \\
\eps_0^{1-\beta} &\leq 1/2,\\
\frac{2 K_1}{K_0 c_\beta \beta} &< \frac{\log \ka}{\ka}.
\end{split}
\end{align}
Then \eqref{eq:boot} holds on the whole interval $[0, T_s)$.
\end{lemma}
\begin{proof} Observe first that \eqref{eq:boot} holds for $\phi(\cdot, 0)$ since $\ka > 1$.
Because of continuity in time, there exists a small time interval $[0, \de)$ in which \eqref{eq:boot}
holds. Suppose \eqref{eq:boot} does not hold on the whole interval $[0, T_s)$, and let
$$
T := \sup\{ t\in [0, T_s): \text{$\eqref{eq:boot}$ holds on $[0, t]$}\} < T_s.
$$
Observe that \eqref{eq:boot} holds on $[0, T)$ and that
the monotonicity of $\phi$ implies that
\beq\label{eq_contr}
\phi(l(T), T) = \ka \eps(T).
\eeq
The first two conditions of \eqref{threecond} allow us to use 
Lemma \ref{lem:lowerBoundEps} to estimate
\begin{align*}
\eps_0^2 K_1 l(T) \int_0^T \frac{ds}{\eps(s)} &= \eps_0^2 K_1 l(T) \int_0^T \frac{-\eps_t(s)}{-\eps_t(s) \eps(s)}\, ds \leq \frac{2 K_1 \ka}{K_0 c_\beta} l(T)  \int_0^T \frac{- \eps_t(s)\,ds}{\eps^{1+\beta}(s)} \\
&\leq \frac{2 K_1 \ka}{K_0 c_\beta \beta} l(T)  (\eps^{-\beta}(T)-\eps_0^{-\beta}) \leq \frac{2 K_1 \ka}{K_0 c_\beta \beta}  (1-(\eps(T)/\eps_0)^{\beta}) \\
&\leq  \frac{2 K_1 \ka}{K_0 c_\beta \beta}.
\end{align*}
Thus by Lemma \ref{lem:boundInnerZone}, 
$$
\eta(l(T), T) \leq \exp\left(\frac{2 K_1 \ka}{K_0 c_\beta \beta}\right).
$$
Now, using the third line of \eqref{threecond},
\bea
\phi(l(T), T)\leq \eps(T) \eta(l(T), T) \leq \eps(T)\exp\left(\frac{2 K_1 \ka}{K_0 c_\beta \beta}\right) < \ka \eps(T),
\eea
contradicting \eqref{eq_contr}.
\end{proof}

\emph{Conclusion of the proof of Theorem \ref{blowupThm}}.
We show that it is possible to choose $0< \beta < 1, 0< \eps_0 < 1, \ka>0$ such that \eqref{threecond}
are satisfied. This is done as follows: Fix $\ka > 2$ and observe that the first and second condition of
\eqref{threecond} together are equivalent to 
\begin{align}\label{cond2}
-\frac{\beta}{1-\beta} \log \ka \leq \log \eps_0 \leq -\frac{\log 2}{1-\beta}.
\end{align}
\eqref{cond2} holds for sufficiently small $\eps_0$ provided $\beta$ can be chosen such that
$-\frac{\beta}{1-\beta} \log \ka < -\frac{\log 2}{1-\beta}$, i.e.
\begin{align}\label{cond3}
\log \ka > \frac{\log 2}{\beta}.
\end{align}
Since $\ka > 2$, \eqref{cond3} holds if $\beta$ is close enough to $1$. Also the third condition of \eqref{threecond}
holds if $\beta$ is close enough to $1$, which follows from $c_\beta\to \infty$ as $\beta\to 1$.

Taking into account \eqref{threecond}, we can now apply Lemma \ref{lemma_5} to see that \eqref{eq:boot} holds on the  
whole existence interval $[0, T_s)$ of the solution. \eqref{lowerBdEps_t} now implies that 
\bea
\eps_t(t) \leq - C \eps^{\beta}(t)
\eea
for some positive, fixed $C>0$, and hence $T_s$ must be finite and $l(T_s)=\eps^\beta(T_s)=0$.

We show now that $\te$ defined by \eqref{def_te} is regular outside the origin, even at time $t=T_s$.
Note that $\phi(x, t)$ for fixed $x$ is monotone nonincreasing in $t$ and $\phi(x, t) \in [0, \eps_0]$ for
$t < T_s$. Hence the pointwise limit
\beq\nonumber
\phi(x, t) = \lim_{t\to T_s} \phi(x, t)
\eeq
exists and is $\geq 0$ everywhere. 
Let $s_0\geq 0$ be the infimum of all numbers $x > 0$ with the property that $\phi(x, T_s) > 0$ (see Figure \ref{f2}).
Note that the set over which the infimum is taken is not empty, since $\phi(x, t) = \eps_0 > 0$ for all $0 \leq t< T_s$ if $x$ is 
outside the support of $g'$.  Moreover, 
\bea\label{lowerbd_Phi}
\Phi(x, t) \geq \eps_0 (x-a) \quad (x \geq a, 0\leq t < T_s)
\eea
if $\operatorname{supp} g' = [0, a]$.

From $\Phi(x, t) = \int_0^x \phi(\si, t)\, d\si$, we see that $\Phi(x, t)$ is nonincreasing in time for fixed $x$
and $\Phi(x, T_s)\geq 0$. Consequently, the pointwise limit $\Phi(x, T_s)$ exists, and
\bea
\text{$\Phi(x, T_s) = 0$ for $0 \leq x \leq s_0$, and $\Phi(x, T_s) > 0$ for all $x > s_0$.}
\eea
Observe that $\Phi(x,T_s)$ is continuous and strictly increasing in $x$ for $x\ge s_0$ since $\phi(x,T_s)>0$ for $x>s_0$.
An elementary argument (using e.g. the fact that $\Phi(x, t)$ is nonincreasing) proves that
as $t\to T_s$, $\Phi^{-1}(y, t)$ converges for all $y > 0$ to a limit $\Phi^{-1}(y, T_s)$. 
The function $\Phi^{-1}(\cdot, T_s)$ is the inverse of $\Phi(\cdot, T_s)$ restricted to the
interval $(s_0, \infty)$. \eqref{def_te} shows that the pointwise limit $\te(y, T_s)$ exists for all $y> 0$.
By Lemma \ref{reg1} below,  $\phi(\cdot, T_s) \in C^1(s_0, \infty)$ and so again by \eqref{rel:basic} $\te(\cdot, T_s)\in C^2(0,\infty)$.
This proves that $\te$ remains regular outside the origin even at the singular time $T_s$.

Finally, we prove the bound \eqref{cuspBound}. First we look at the case $s_0 = 0$. A key observation
is that from \eqref{upperBd} we get the upper bound
\begin{align*}
\Phi(z, T_s) \leq C(\beta, \ka) z^{\frac{1}{\beta}+1} \quad (0 \leq z < \infty),
\end{align*}
implying
\beq\label{eq_PhiInv}
\Phi^{-1}(y, T_s) \geq \widetilde C(\beta, \ka) y^{\frac{\beta}{\beta+1}}.
\eeq
From \eqref{rel:basic} we get, using the substitution $\Phi^{-1}(\tilde y, t) = z$,
\begin{align}\nonumber
\te(y, t) = \te(0, t) + \int_0^y \frac{g'(\Phi^{-1}(\tilde y, t))}{\phi(\Phi^{-1}(\tilde y, t))} d\tilde y = \te_0(0) + \int_0^{\Phi^{-1}(y, t)} g'(z)\, dz
\end{align}
Now pass to the limit $t\to T_s$ and continue the estimate by using first \eqref{eq_K1} and then \eqref{eq_PhiInv},
\begin{align}\nonumber
\int_0^{\Phi^{-1}(y, T_s)} g'(z)\, dz \leq  - C(\eps_0, K_0, \beta, \ka) y^{\frac{2\beta}{\beta+1}}
\end{align}
for all $y\in [0, \Phi(1, T_s)]$. Note that $\Phi(1, T_s)>0$ because of $s_0=0$.

Since $\te(y, T_s)\leq \te(\Phi(1, T_s), T_s)$ for $y \geq \Phi(1, T_s)$ 
we can define $\nu = \frac{2\beta}{\beta+1} < 1$ and adjust the constant such that \eqref{cuspBound}
holds (note that $\te(y, T_s)$ has compact support).

If $s_0 > 0$, we note $\sup_{\R+}\theta(y, T_s)  < \te(0, T_s) = \te_0(0)$. Hence, the constant $c$ can be adjusted such that 
\eqref{cuspBound} holds.
This concludes the proof of Theorem \ref{blowupThm}.

\begin{lemma}\label{reg1}
Suppose $\lim_{t\to T_s}\Phi(x_0, t) > 0$. Then $\phi(\cdot, T_s) \in C^1[x_0, \infty)$.
\end{lemma}
\begin{proof}
Since for fixed $t$, $\Phi(x, t)$ is non-decreasing in $x$, $\Phi(x, t) \geq \Phi(x_0, t)>0$ for all $x \geq x_0$, $t < T_s$. Using this lower bound for $\Phi$,
the $C^1$-norm of the right-hand-side of equation \eqref{main_eq} on
the interval $[x_0, \infty)$ can be bounded:
\bea
\|\phi_t(\cdot, t)\|_{C^1[x_0, \infty)} \leq C(x_0)
\eea
for all $t< T_s$, where $C(x_0)$ depends on $x_0$ but not on $t$. This implies the convergence $\phi(\cdot, t)\to \phi(\cdot, T_s)$ 
in $C_1([x_0, \infty))$ as $t\to T_s$.
\end{proof}

\begin{remark}\label{why_eps_0}
It is not necessary to choose $(\eps_0, g(\cdot; \eps_0))$ with sufficiently small $\eps_0$
for the pair $(\phi_0, g)$. In fact, we could have worked with the canonical choice $(1, \te_0)$. 
In this case, one first proves a bound of the form $-\eps_t \leq k \eps$ with $k>0$. This means
that after some positive time $T_0$, $\eps(t)$ is small enough such that \eqref{threecond} can
be satisfied.
\end{remark}

\begin{remark}
For the moment, we leave the question open if a needle-like discontinuity can really form, or if generically a cusp is 
obtained at the singular time. In this context, it is interesting to observe that a suitable lower bound 
\beq
\phi(x, T_s) \geq c_0 x^{p}
\eeq
with $c_0>0$, $p>1$ would suffice to exclude the needle scenario, and would nicely complement \eqref{upperBd}.
A lower bound for $\phi(\cdot, t)$ for all $t < T_s$ was obtained by A.~Zlato\v{s} \cite{ZlatosPer}, however,
his lower bound does not give any information in the limit $t\to T_s$.
\end{remark}

\section{Acknowledgements}
The authors would like to thank Alexander Kiselev for helpful
discussions and suggestions greatly improving the presentation. 
We thank Andrej Zlato\v{s} for helpful discussions and his remarks on a lower bound for
the function $\phi$. VH expresses his gratitude to
the German Research Foundation (DFG) for continued support through grants FOR 5156/1-1 and FOR 5156/1-2,
without this research could not have been undertaken. VH also acknowledges partial support by NSF grant NSF-DMS 1412023. 

\end{document}